\documentclass[11pt,reqno]{amsart}
\usepackage{amsaddr}  
\usepackage[T1]{fontenc}
\usepackage[notrig]{physics}
\usepackage{tikz}
\usepackage{graphicx,enumerate,amssymb,bbold}
\usepackage[ruled, linesnumbered, noend,vlined]{algorithm2e}
\usepackage[a4paper, left=1.4in, right=1.4in, top=1.6in, bottom=1.3in]{geometry}
\usepackage{float}
\usepackage{inputenc}
\usepackage{bbm}
\usepackage{subcaption}
\usepackage{placeins}
\usepackage{enumitem}
\usepackage{parskip}
\usepackage{setspace}
\usepackage{microtype}

\usepackage[colorlinks=true, pdfstartview=FitV, linkcolor=blue, 
            citecolor=blue, urlcolor=blue]{hyperref}
\usepackage{cleveref}

\numberwithin{equation}{section}
\numberwithin{figure}{section}
\definecolor{mLightBrown}{HTML}{EB811B}

\theoremstyle{plain}
\newtheorem{theorem}{Theorem}[section]
\crefname{theorem}{Theorem}{Theorems}
\newtheorem{lemma}[theorem]{Lemma}
\crefname{lemma}{Lemma}{Lemmata}
\newtheorem{proposition}[theorem]{Proposition}
\crefname{proposition}{Proposition}{Propositions}
\newtheorem{corollary}[theorem]{Corollary}
\crefname{collary}{Corollary}{Corollaries}

\newtheorem{remark}{Remark}

\newcommand{\bitem}{\begin{itemize}}
\newcommand{\eitem}{\end{itemize}}
\newcommand{\mc}[1]{\mathcal{#1}}

\newcommand{\N}{\mathbb{N}}
\newcommand{\R}{\mathbb{R}}

\newcommand{\bpm}{\begin{pmatrix}}
\newcommand{\epm}{\end{pmatrix}}
\newcommand{\bsm}{\left(\begin{smallmatrix}}
\newcommand{\esm}{\end{smallmatrix}\right)}

\newcommand{\la}{\langle}
\newcommand{\ra}{\rangle}

\newcommand{\vphi}{\varphi}

\newcommand{\superPoi}{{_\text{\(^{\Poi}\)}}}
\newcommand{\superIP}{{_\text{\(^{\IP}\)}}}
\newcommand{\xk}{x^{(k)}}
\newcommand{\xkp}{x^{(k+1)}}

\newcommand{\xt}{x(\tau)}
\newcommand{\Rd}{\nabla_{\mc{M}}}

\newcommand{\eins}{\mathbb{1}}

\DeclareMathOperator{\Diag}{Diag}

\DeclareMathOperator{\dom}{dom}
\DeclareMathOperator{\intr}{int}

\DeclareMathOperator{\argmin}{arg min}

\DeclareMathOperator{\Hess}{Hess}

\DeclareMathOperator{\Exp}{Exp}

\DeclareMathOperator{\Poi}{Poi}
\DeclareMathOperator{\IP}{IP}

\DeclareMathOperator{\KL}{KL}

\title{
Information Geometry of Exponentiated Gradient: Convergence beyond L-Smoothness\textsuperscript{\ddag}}

\author{ \small{Yara Elshiaty\textsuperscript{*, \dag},  Ferdinand Vanmaele\textsuperscript{\dag}, \and Stefania Petra\textsuperscript{\dag}}}

\begin{document}

\begin{abstract}
We study the minimization of smooth, possibly nonconvex functions over the positive orthant, a key setting in Poisson inverse problems, using the exponentiated gradient (EG) method. Interpreting EG as Riemannian gradient descent (RGD) with the 
$e$-Exp map from information geometry as a retraction, we prove global convergence under weak assumptions -- without the need for $L$-smoothness -- and  finite termination of Riemannian Armijo line search. Numerical experiments, including an accelerated variant, highlight EG's practical advantages, such as faster convergence compared to RGD based on interior-point geometry.
\end{abstract}

\maketitle
\renewcommand{\thefootnote}{\fnsymbol{footnote}}
\footnotetext[3]{This preprint has not undergone peer review or any post-submission improvements or corrections. A Version of Record of this contribution will appear in the proceedings of the Scale Space and Variational Methods in Computer Vision (SSVM) 2025 conference, to be published in the Lecture Notes in Computer Science (LNCS) series by Springer. The final published version will be available online.
}
\footnotetext[1]{Institute for Mathematics, Heidelberg University \; (\url{elshiaty@math.uni-heidelberg.de})}
\footnotetext[2]{Institute for Mathematics \& Centre for Advanced Analytics and Predictive Sciences (CAAPS), University of Augsburg \; (\url{ferdinand-joseph.vanmaele@uni-a.de},\url{stefania.petra@uni-a.de})}

\markboth{\MakeUppercase{Yara Elshiaty, Ferdinand Vanmaele, and Stefania Petra}}{\MakeUppercase{Information Geometry of EG: Convergence beyond L-Smoothness}}

\section{Introduction} 
We consider the problem of minimizing a smooth potentially \emph{nonconvex} function on the positive orthant, common in Poisson inverse problems, nonnegative sparse coding, and tomographic reconstruction. Specifically, we aim to solve the optimization problem
\begin{equation*}
    f^{\min} := \min_{x \in \R^n_{+}} f(x)
\end{equation*}
where $f$ is smooth on $\intr \dom f=\R^n_{++}$, and $\R^n_{++}$ denotes the $n$-dimensional positive orthant, assuming $f^{\min} > - \infty$. 
Using information geometry, we explore the Riemannian structure of the parameter manifold of the Poisson distribution, which corresponds to the positive orthant. From this perspective, the \emph{exponentiated gradient} (EG) method can be interpreted as Riemannian gradient descent (RGD), where line search is performed along appropriately chosen geodesics.
The EG updates are given by
\begin{equation} \label{eq:EG} \tag{EG}
    x^{(k+1)} := x^{(k)}(\tau_k) := x^{(k)} \cdot \exp(- \tau_k \nabla f(x^{(k)})), \quad \forall k \in \N
\end{equation}
for an initial point $x^{(0)} \in \R^n_{++}$, whereas $\tau_k>0$ denotes the step size. While RGD methods ensure convergence to a local minimum with Riemannian Armijo line search if an accumulation point exists, this assumption is nontrivial. 

\vspace{2mm}
\noindent
\textbf{Related work.} 
The \eqref{eq:EG} method \cite{kivinen1997exponentiated} is a special case of \emph{mirror descent (MD)}
\begin{equation}\label{eq:BPGD}
    x^{(k+1)} = \argmin_{x \in \mc{M}} \tau_k \la \nabla f(x^{(k)}), x - x^{(k)} \ra + \KL(x, x^{(k)})
\end{equation}
with the Kullback-Leibler divergence $\KL(x,y) = \big \la x, \log \frac{x}{y} \big \ra - \la \eins, x - y \ra$ for $x \in \R^n_+, y \in \R^n_{++}$
and has been studied in various contexts, particularly for probability simplex constraints, where its iteration rule is computationally efficient, avoiding the need for projection.
Convergence guarantees typically require Lipschitz continuity of the objective function or gradient \cite{Auslender2006,Beck_mirror}  or smoothness with respect to the negative entropy \cite{Bauschke:2017aa}, with convergence achieved using either a fixed step size or Armijo line search.
However, these conditions are often violated in practical applications (see \cref{sec:experiments}). 
Notably, convergence guarantees under weak conditions have been shown for normalized EG with Euclidean Armijo line search on quantum density matrices \cite{Li_Cevher:2019}. 
Building on these results, we extend the analysis in \cite{Li_Cevher:2019} to the positive orthant using Riemannian Armijo line search within an information-geometric framework.
The relation of EG and RGD was studied in \cite{Raskutti2015,Raus2024}.

\textbf{Contribution.}
We establish the global convergence of EG under weak conditions, providing a convergence guarantee for EG as a RGD method with \emph{Riemannian Armijo line search}. This result relies solely on the smoothness of the cost function and the finite termination of the line search, which leverages the self-concordant-like properties of the Kullback-Leibler divergence. Notably, this condition is substantially weaker than the relative $L$-smoothness assumption \cite{Bauschke:2017aa} commonly required for MD convergence.

\textbf{Organization.}
We introduce the Poisson geometry on the positive orthant, interpreting EG updates as RGD with the $e$-Exp map as a retraction in \cref{sec:Geometry-EG}. 
\cref{sec:EG-R-Armijo} proves finite termination of Riemannian Armijo line search under weak assumptions (\cref{thm:RArmijo-terminates}) and examines acceleration via geometric conjugate gradient directions. \cref{sec:experiments} showcases EG's practical advantages in Poisson inverse problems, including faster convergence over interior-point RGD.

\textbf{Basic Definitions and Notation.}
Let $\la \cdot, \cdot \ra$ denote the Euclidean inner product on $\R^n$. For a sufficiently smooth function $f$, we denote its gradient by $\nabla f$ and its Hessian by $\nabla^2 f$. We write $xy = (x_1y_1, \dots,x_n y_n)^T$ for the Hadamard product of two vectors $x, y \in \R^n$ and componentwise division by $\frac{x}{y}$ if $y \neq 0$. Like-wise we define the functions $\exp(x)$, $\log(x)$ to a vector $x$ elementwise.
For a smooth manifold $\mc{M}$, the \emph{tangent space} at $x \in \mc{M}$ is denoted by $T_x\mc{M}$. 
On a Riemannian manifold $(\mc{M}, g)$, the \emph{Riemannian metric tensor} $g$ defines an inner product $\la u, v \ra_{g(x)} := \la u, g(x) v \ra$ for $u, v \in T_x\mc{M}$. When the context is clear, we simply write $\la u, v \ra_x$.
Given an \emph{affine connection} $\mc{D}$, the \emph{exponential map} $\Exp$ is defined by $\Exp_x(v) = \gamma_{x, v}(1)$, where $\gamma_{x, v}(\tau)$ is the $\mc{D}$-geodesic through $x = \gamma_{x, v}(0)$ with initial velocity \(\dot{\gamma}_{x, v}(0) = v\). We call the Levi-Civita connection the $g$-connection for short. For further details, see \cite{Amari:2000,Jost2005}.
\section{Geometry of EG on the Positive Orthant}\label{sec:Geometry-EG}

\textbf{The Poisson Geometry.}
For a discrete random vector $z \in \N^n$, we define the vector of density functions for the Poisson distribution as
\begin{equation*}
    \tilde p_i(z; x) := \frac{x_i^{z_i} \exp(-x_i)}{z_i!}, \qquad x \in \R^n_{++}.
\end{equation*}
This can be rewritten as an exponential family with the \emph{e-parameters} $x_\ast$
\begin{equation*}\label{eq:Poisson-e-distribution}
    p_i(z; x_\ast) = \frac{1}{z_i!} \exp(z_ix_{\ast,i} - \psi^*_i(x_{\ast})),
\end{equation*}
where $\psi^\ast(x_\ast) = \la \eins, e^{x_\ast} \ra$ is the log-partition function. By the classical Legendre transform $\nabla \psi^\ast = (\nabla \psi)^{-1}$, we define the dual parameters as
\begin{equation}\label{eq:def-x-xast}
    x_{\ast} := \nabla \psi(x) = \log x,\qquad
    x = \nabla \psi^{\ast}(x_{\ast}) = e^{x_{\ast}},\quad \text{with} \; \;
    \intr \dom \psi = \R^n_{++},
\end{equation}
where the convex conjugate function of $\psi^\ast$ is given by the negative entropy:
\begin{equation} \label{eq:def-negative-entropy}
    \psi(x) := \sum_i \psi_i(x) := \la x \log x \ra - \la 1, x \ra, \qquad x \in \R^{n}_{++}.
\end{equation}
We define a Riemannian structure $(\mc{M}, g)$ on $\mc{M} = \intr \dom \psi = \R^n_{++}$ as the $n$-product manifold equipped with the Fisher-Rao metric
\begin{equation} \label{eq:Poisson-geometry}
    g(x) = \nabla^2 \psi = \Diag\Big( \frac{1}{x}\Big), \qquad g_{ij}(x) = \frac{\partial^2}{\partial x_i \partial x_j} D_{\psi}(x,y)\big|_{y = x},
\end{equation}
where $D_\phi(x,y) = \phi(x) - \phi(y) - \la \nabla \phi(y), x - y \ra$ for $x \in \dom \phi, y \in \intr \dom \phi$ denotes the \emph{Bregman divergence} induced by a convex function $\phi$. For the Poisson distribution, the Bregman divergence induced by $\psi$ corresponds to the KL divergence.
We naturally identify $T_x\mc{M} \simeq \R^n$ for all $x \in \mc{M}$.
This defines a dually flat Riemannian structure $(\mc{M}, g, \mc{D}^m, \mc{D}^e)$ with the \emph{m-connection} $\mc{D}^m$ and \emph{e-connection} $\mc{D}^e$ induced as the primal and dual connections, respectively, cf. \cite{Amari:2000}. Henceforth, we will be working with $\mc{D}^e$. Its Christoffel symbols of first kind are given by
\begin{equation} \label{eq:e-Poi-Christoffel-symbols}
    \Gamma_{ij,k} = \partial_i \partial_j \partial_k \psi(x) = \begin{cases}
        - \frac{1}{x_i^2},& \text{if } i=j=k \\ 0,& \text{otherwise},
    \end{cases}
\end{equation}
with the second kind yielding $\Gamma_{kk}^k = g^{kk}(x) \Gamma_{kk,k} = - \frac{1}{x_k}$ and zero otherwise. This enables us to solve the decoupled ODEs defining the e-geodesics $\gamma^e_{x,v}(\tau) = (\gamma_1(\tau), \dots ,\gamma_n(\tau))$, \cite[Def.~1.4.2]{Jost2005}
\begin{equation*} 
    \frac{d^2\gamma_k}{d\tau^2}-\frac{1}{\gamma_k} \Big(\frac{d\gamma_k}{dt}\Big)^2 = 0, \quad \forall k \in [n]
\end{equation*}
in closed form. As a result, the geodesic curve emanating from $\gamma^e_{x,v}(0)=x \in \mc{M}$ in direction $\dot\gamma^e_{x,v}(0)=v$ reads
\begin{equation*} 
\gamma^e_{x,v}(\tau) = x \exp\Big(\tau \frac{v}{x}\Big).
\end{equation*}
Thus, the $e$-exponential map $\Exp^e_x: v \mapsto \gamma^e_{x,v}(1) = x \exp(\frac{u}{x})$ is complete. 

\vspace{2mm}
\noindent
\textbf{EG as Riemannian gradient descent.}
Let $f: \mc{M} \to \R$ be a smooth function.
By definition \cite[Eq.~(3.31)]{Absil:2008aa}, the \emph{Riemannian gradient} $\nabla_{\mc{M}} f(x)$ of the objective function $f$ is the unique tangent vector satisfying
\begin{equation*} 
\la \nabla_{\mc{M}} f(x), v \ra_{g(x)} = \la \nabla f(x), v \ra,\qquad
\forall v \in T_{x}\mc{M}.
\end{equation*}
A point $x \in \mc{M}$ is a \emph{critical point} of $f$ if and only if $\nabla_{\mc{M}_g} f(x) = 0.$
Consequently, in the context of the Poisson geometry, the Riemannian gradient is given by
\begin{equation*}
    \Rd f(x) \overset{\eqref{eq:Poisson-geometry}}{=} \Diag(x) \nabla f(x)
\end{equation*}
and $x \in \mc{M}$ is a critical point of a smooth function $f$ if and only if $\nabla f(x) = 0$, since $x \neq 0$.

The RGD method for a given exponential map $\Exp$ defines the iterative update rule
\begin{equation} \label{eq:RGD}
    \tag{RGD}
    \xkp = \Exp_{\xk}\Big(-\tau_k \nabla f_{\mc{M}}(\xk)\Big)
\end{equation}
with step size $\tau_k > 0$.
The definition of the Riemannian gradient on the Poisson manifold leads to the following representation of the EG iteration.
\begin{proposition}[EG as RGD {\cite[Proposition 4.1]{Raus2024}}]
    Let $\mc{M} = \R^n_{++}$ be endowed with the Poisson Fisher-Rao metric \eqref{eq:Poisson-geometry}. Then, the \eqref{eq:EG} iteration is equivalent to
    \begin{equation*}
        x^{(k+1)} = \Exp_{x^{(k)}}^{e}\big(-\tau_k \nabla_{\mc{M}_{\superPoi}}f(x^{(k)})\big).
    \end{equation*} 
\end{proposition}
We refer to the Riemannian gradient descent step using the $e$-exponential map and the Poisson Fisher-Rao metric as Poisson $e$-RGD, and use it interchangeably with EG. 
\begin{remark}
    Note that the Poisson geodesic equations are also realized by other geometries on the positive orthant. In this remark, we briefly explore a notable case. The \emph{interior-point geometry} defined by the Riemannian metric
    \begin{equation*} 
        g_\superIP(x) = \Diag\Big( \frac{1}{x^2}\Big),
    \end{equation*}
    can be constructed as the Hessian of the barrier function $- \log x$ (cf. \cite{NesterovT02}). This geometry corresponds to the one generated by the Fisher-Rao metric tensor of the exponential distribution. Since $g_\superIP$ is diagonal, the Christoffel symbols for the $g$-connection are given by
    \begin{equation*} 
        \Gamma_{kk}^k(x) = \frac{1}{2} g_\superIP^{kk} \partial_k {(g_\superIP)}_{kk} = - \frac{1}{x_k}.
    \end{equation*}
    They coincide with the Christoffel symbols for the $e$-connection with respect to the Poisson metric, cf. \eqref{eq:e-Poi-Christoffel-symbols}, meaning that for all points $x \in \mc{M}$ and tangents $v \in T_x\mc{M}$
    \begin{equation*} 
         \gamma_{x,v}^{g_\superIP}(\tau) = \gamma_{x, v}^{e_\superPoi}(\tau) =x \exp\Big(\tau\frac{v}{x}\Big),
    \end{equation*}
    where we added superscripts to distinguish the two geometries. 
    However, the EG update cannot be reformulated as the $g$-RGD on the Riemannian manifold $(\mc{M}, g_\superIP)$ since ${\Rd}_{\superIP}f(x) = x^2 \nabla f(x)$. The RGD iterates of the $(\mc{M}, g_\superIP)$ with the $g$-Exp map are then given by
    \begin{equation} \label{eq:IP-g-RGD}
        x^{(k+1)} = x^{(k)} \exp(- \tau_k x \nabla f(x)),
    \end{equation}
    which we denote by IP$g$-RGD when necessary, see \cref{sec:experiments}.
\end{remark}

\section{EG with Riemannian Armijo line search}\label{sec:EG-R-Armijo}
Given a point $x \in \mc{M}$ and constants $\bar \tau > 0, \beta, \sigma \in (0,1)$ the Armijo line search outputs $\tau_k = \beta^m \bar \tau$ where $m$ is the least nonnegative integer that satisfies
\begin{equation} \label{eq:R-Armijo}
    \tag{R-Armijo}
    f\Big(\Exp_x\big(-\beta^m \bar \tau \nabla_{\mc
    M} f(x)\big)\Big) 
    \leq f(x) - \sigma \beta^m \bar \tau \| \nabla_{\mc{M}}f(x) \|_x^2.
\end{equation}
\cite[Proposition 4.7]{Boumal2023} shows that every accumulation point of the Armijo backtracking process is a critical point of 
$f$. However, without additional smoothness assumptions, such as Lipschitz continuity of $f$ or its gradient, finite termination of backtracking — and thus the existence of accumulation points — cannot be guaranteed \cite[Lemma 4.12]{Boumal2023}.

We establish the termination of EG, framed as Poisson e-RGD with Armijo-based backtracking. Recall, that $\xt = \Exp_x^{e_\superPoi}(- \tau \Rd f(x))$.

\begin{theorem}[Termination of \eqref{eq:R-Armijo}]\label{thm:RArmijo-terminates}
    Let $f: \mc{M} \to \R$ be a smooth function. Then there exists some $\tau_x > 0$ such that
    \begin{equation*}
        f(\xt) \leq f(x) - \sigma \tau \| \Rd f(x)\|^2_x \quad \forall \tau \in [0, \tau_x].
    \end{equation*}
\end{theorem}
The following convergence statement is a direct consequence of the termination of the backtracking process and \cite[Proposition 4.7]{Boumal2023}.
\begin{corollary}[Convergence of EG with Riemannian Armijo line search] \label{cor:convergence-loc-min}
The sequence $f(x^{(k)})$ defined by \eqref{eq:RGD} and step size \eqref{eq:R-Armijo} monotonically converges to a local minimizer of $f$. If $f$ is convex, the sequence converges to a global minimizer.
\end{corollary}
We follow the proof of \cite[Proposition 3.2]{Li_Cevher:2019}, which establishes a similar result for the normalized exponentiated gradient on the set of quantum density matrices. The proof relies crucially on the properties of $h$ which characterizes the Kullback-Leibler divergence $\KL = D_\psi$, for $\psi$ given by \eqref{eq:def-negative-entropy}, in the following Lemma.
\begin{lemma} \label{lm:Kl-via-SCL}
    For any $x \in \mc{M}$ and $\tau \geq 0$ the following holds
    \begin{equation*}
        \KL(\xt, x) = D_h(0, \tau)
    \end{equation*}
    with $h(\tau) := \la \eins, \xt \ra$.
\end{lemma}
\begin{proof}
    This follows by the identity $h(\tau) = \psi^\ast(\xt_\ast) = \psi^\ast(\log \xt)$ as defined in \eqref{eq:def-x-xast} and the duality
    \begin{equation*}
        \KL(\xt, x) = D_\psi(\xt, x) = D_{\psi^\ast}( \log x, \log \xt) = D_h(0, \tau). 
    \end{equation*}
\end{proof}
 We call a three times continuously differentiable function $d: \R \to \R$ $\mu$-\emph{self-concordant like} if and only if it is convex and satisfies $|d^{\prime \prime \prime}(\tau)| \leq \mu d^{\prime \prime}(\tau)$ for all $\tau$. Once can easily see, that $h$ is convex and especially $\mu$-self-concordant like with $\mu = \|\nabla f\|_\infty$ which is bounded by the continuity of $\nabla f$. The derivatives up to third order then state
\begin{align*}
    h^\prime(\tau) &= \la - \nabla f(x), x(\tau) \ra, \\
    h^{\prime \prime}(\tau) &= \la \big( \nabla f(x)\big)^2, x(\tau) \ra, \\
    h^{\prime \prime \prime}(\tau) &= -\la \big(\nabla \big(f(x)\big)^3, x(\tau) \ra,
\end{align*}
whereas we use the identity $\dot x(\tau) = - \nabla f(x)\xt.$
The representation of the Poisson geometry-defining $\KL$-divergence by a self-concordant-like function is the crucial step for the proof of \cref{thm:RArmijo-terminates}. It enables the following result which we will make use of later.
\begin{lemma}[KL scaling property]\label{lm:Cor3.1}
    Let $x \in \mc{M}$ and $\bar \tau > 0$. Suppose $\mu > 0$. Then it holds that
    \begin{equation*}
        \kappa \KL(x(\bar \tau), x) \leq \frac{\KL(\xt, x)}{\tau^2}, \quad \forall \tau \in (0, \bar \tau],
    \end{equation*}
    where $\kappa := \frac{\mu^2}{2} \big(\exp(\mu \bar \tau) (\mu \bar \tau - 1)+1\big)^{-1}$.
\end{lemma}
\begin{proof}
    The statement is a consequence of the equivalent formulation of \cite[Proposition 3.1]{Li_Cevher:2019} with \cref{lm:Kl-via-SCL} and the self-concordant-likeness of $h$.
\end{proof}
Lastly, recall the approximation of general Bregman divergences
\begin{equation} \label{eq:quad-approx-Dvphi}
        D_\phi(y,x) = \frac{1}{2} \sum_{i,j} g_{ij}(y) (y_i - x_i)(y_j - x_j) + o(\|y-x\|^2)
\end{equation}
for $g_{ij}(x) = \frac{\partial^2}{\partial x_i \partial x_j} D_\phi(x,z) _{z = x}$, cf. \cite[Chapter 3.2]{Amari:2000}. For the $\KL$-divergence we get the following corollary.
\begin{corollary}[Pinsker's type inequality for positive orthant]\label{cor:1norm-Kl-Upper-bound}
    For $x \in \mc{M}$ and $\tau_x$ sufficiently small, there exists a constant $c > 0$ such that
    \begin{equation*}
        \|\xt - x \|_1 \leq \sqrt{2c} \sqrt{\KL(\xt, x)} \quad \forall \tau \in [0, \tau_x]
    \end{equation*}
\end{corollary}
\begin{proof}
    Since $\KL$-divergence induces the Poisson geometry we get
    \begin{equation*}
        \KL(y,x) \overset{\eqref{eq:quad-approx-Dvphi}}{\geq} \frac{1}{2} \|y-x\|^2_{g(y)} \geq  \frac{1}{2} \frac{\|y - x\|_1^2}{\|y\|_1},
    \end{equation*}
    whereas we use Jensen's inequality for the second inequality. Set $y = \xt$. Since $x(\tau)$ is continuous and we take $\tau$ on a closed interval, there exists a constant $c > 0$ that upper bounds $\|\xt\|_1 \leq c$ and we obtain the statement by rearranging the terms.
\end{proof}
Before we prove \cref{thm:RArmijo-terminates}, we remind that the formulation of EG as an MD update \eqref{eq:BPGD} implies
\begin{equation}\label{eq:LemmaB.1}
    \la \nabla f(x), \xt - x \ra \leq - \frac{D_\psi(\xt, x)}{\tau} = - \frac{\KL(\xt, x)}{\tau}.
\end{equation}
\begin{proof}[of Thm.~\ref{thm:RArmijo-terminates}]
    Since $x = \xt$ for $\mu = 0$ and thus the line search terminates, assume $\mu > 0$ henceforth.
    The statement is equivalent to 
    \begin{equation*} \label{eq:RArmijo-form2}
        f(\xt) - f(x) \leq - \sigma \tau h^{\prime \prime}(0), \quad \forall \tau \in [0, \tau_x],
    \end{equation*}
    since 
    \begin{equation*}
        \|\nabla_{\mc{M}} f(x) \|^2_x = \la x \nabla f(x), \nabla f(x) \ra = \la (\nabla f(x))^2, x(0) \ra = h^{\prime \prime}(0).
    \end{equation*}
    For $\tau > 0$ by the mean value theorem there exists a point $y$ on the euclidean segment between $\xt$ and $x$ such that
    \begin{equation*}
        f(\xt) - f(x) = \la \nabla f(y), \xt - x \ra.
    \end{equation*}
    By adding a zero and rearranging, the statement can be reformulated as
    \begin{equation*}
        \la \nabla f(y) - \nabla f(x), \xt - x \ra + \sigma \tau h^{\prime \prime}(0) \leq - \la \nabla f(x), \xt - x \ra \quad \forall \tau \in [0, \tau_x].
    \end{equation*}
    By \eqref{eq:LemmaB.1} it would suffice to show
     \begin{equation} \label{eq:Thm2-equiv-formulation}
        \la \nabla f(y) - \nabla f(x), \xt - x \ra + \sigma \tau h^{\prime \prime}(0) \leq \frac{\KL(\xt, x)}{\tau} \quad \forall \tau \in [0, \tau_x].
    \end{equation}
    Set $\eta := \min_{i \in [n]} |\partial_i f(x) |$. For $\tau_x$ small enough, the Taylor expansion of $\xt$ yields
    \begin{equation*}
        \|\xt - x \|_1 \geq \tau \|x \nabla f(x) \|_1 \geq \tau \eta \|x\|_1 \quad \forall \tau \in [0, \tau_x].
    \end{equation*}
    With this, we upper bound $h^{\prime \prime}(0)$:
    \begin{equation*}
        \tau h^{\prime \prime}(0) \leq \tau \mu^2 \|x\|_1 \leq \frac{\mu^2}{\eta} \|\xt - x \|_1.
    \end{equation*}
    Furthermore, we use Hölder's inequality to get
    \begin{align*}
        \la \nabla f(y) - \nabla f(x), &\xt - x \ra + \sigma \tau h^{\prime \prime}(0)  \\ 
        &\leq \la \nabla f(y) - \nabla f(x), \xt - x \ra + \sigma \frac{\mu^2}{\eta} \|\xt - x \|_1 \\
        &\leq \Big(\| \nabla f(y) - \nabla f(x) \|_\infty + \sigma \frac{\mu^2}{\eta} \Big)\|\xt - x \|_1 \\
        &\overset{\text{Cor.}~\ref{cor:1norm-Kl-Upper-bound}}{\leq} \sqrt{2c} \Big(\| \nabla f(y) - \nabla f(x) \|_\infty + \sigma \frac{\mu^2}{\eta} \Big) \sqrt{\KL(\xt, x)}
    \end{align*}
    for the left hand side of the inequality \eqref{eq:Thm2-equiv-formulation}.
    \begin{equation*}
        \sqrt{\KL(\xt, x)} \sqrt{\kappa \KL(x(\bar \tau), x)} \leq \frac{\KL(\xt, x)}{\tau} \quad \forall \tau \in [0, \tau_x]
    \end{equation*}
    for some $\kappa > 0$. Thus, the statement follows if
   \begin{equation*}
        \sqrt{2c}\Big( \| \nabla f(y) - \nabla f(x) \|_\infty + \sigma \frac{\mu^2}{\eta} \Big) \leq \sqrt{\kappa \KL(x(\bar \tau), x)} \quad \forall \tau \in [0, \tau_x].
    \end{equation*}
    Since $\nabla f$ is continuous and the right hand side is strictly positive, the inequality holds for a small enough $\tau_x$, proving the statement.
\end{proof}

\noindent
\textbf{Exactness of the Riemannian Armijo condition.}
The output of \eqref{eq:R-Armijo} aims to approximate the exact solution of the line search $\tau_x^{\text{ex}}$ satisfying
\begin{equation*}
    \tau_x^{\text{ex}} = \argmin_{\tau \in \R} f(\xt). 
\end{equation*}
The first optimality condition $\frac{d}{d\tau} f(\xt) \big|_{\tau = \tau_x^{\text{ex}}} = 0$ yields the implicit form
\begin{equation} \label{eq:OC-exact-line-search}
    -\la \Rd f(x), \Rd f((x_{\tau_x^\text{ex}})) \ra_x = 0,
\end{equation}
since
\begin{align*}
    \frac{d}{d\tau} f(\xt) &= \la \nabla f(\xt), \dot x(\tau) \ra = - \la \nabla f(\xt), \nabla f(x) \xt \ra \\
    &= - \Big \la \xt \nabla f(\xt), \frac{x}{x} \nabla f(x) \Big\ra = - \la \Rd f(x), \Rd f(\xt) \ra.
\end{align*}
For $\tau \geq 0$, define $\Delta_x(\tau) :=  -\la \Rd f(x), \Rd f(\xt) \ra_x$. The following lemma relates the optimality condition of the exact line search to second-order information of $f$.

\begin{lemma} \label{lm:exactness-approx}
    For $\tau$ sufficiently small, the exactness condition \eqref{eq:OC-exact-line-search} is approximated by
    \begin{equation*}
        \Delta_x(\tau) = - \|\Rd f(x) \|_x^2 + \tau \la \Rd f(x), H(x) \ra_x + \mathcal{O}(\tau^2)
    \end{equation*}
    with  $H(x) = x (\nabla f(x))^2 + x \nabla^2f(x)[\Rd f(x)]$.
\end{lemma}

\begin{proof}
    The statement is a consequence of Taylor expansion of $\Rd f(\xt) = \xt \nabla f(\xt)$ for small enough $\tau$. A direct calculation yields
    \begin{align*}
        \Rd f(\xt) = &= x \nabla f(x) + \tau \Big(\dot x(\tau)\nabla f(\xt) + \xt \nabla^2f(\xt)[\dot x(\tau)] \Big) \Big|_{\tau = 0} + \mc{O}(\tau^2) \\
    &= \Rd f(x) - \tau H(x) + \mc{O}(\tau^2),
    \end{align*}
    with $\dot x(\tau) = - \xt \nabla f(x)$. The claim follows by insertion in $\Delta_x(\tau)$.
\end{proof}
Consequently, the following corollary follows by rearrangement. 
\begin{corollary}[Exactness of \eqref{eq:R-Armijo}] \label{cor:RArmijo-exactness}
    If $\tau$ satisfies \eqref{eq:R-Armijo} at the point $x \in \mc{M}$, then
    \begin{equation*}
        f(\xt) - f(x) \leq \tau \Delta_x(\tau)  - \tau^2 \la \Rd f(x), H(x) \ra_x + \mc{O}(\tau^3).
    \end{equation*}
\end{corollary}

\begin{remark}
    While first derivatives exist, construction of second derivatives (Hessians) on Riemannian manifolds is dependent on a choice of a connection, and as such is not canonical, see \cite{Jost2005}. Thus, there are various notions of generalizing the Euclidean Hessian \cite{Boumal2023,Absil:2008aa}, which are equivalent for the $g$-connection but lead to differing definitions when extended to other types of connections.
    We choose the definition $\Hess f(x)[v] = \mc{D}_v \Rd f(x)$ for a given connection $\mc{D}$. Thus, we can write the tangent vector $H(x)$ as the $m$-Hessian
    \begin{equation*}
        \Hess^m f(x)[\Rd f(x)] = x (\nabla f(x))^2 + x \nabla^2f(x)[\Rd f(x)].
    \end{equation*}
    The definition of a Riemannian Hessian is used in optimization to establish a notion of geodesic convexity, \cite{Absil:2008aa,Boumal2023}. An examination of the consequence of geodesic convexity with respect to a non Riemannian-connection and the performance of the Armijo line search are beyond the scope of this paper.
\end{remark}

\noindent
\textbf{Accelerating EG.}
We shortly highlight a standard acceleration method for RGD, applied to the EG update. While RGD only exploits steepest descent direction at the current iteration, the geometric conjugate gradient (CG) descent method updates the classical descent direction $v_k = - \Rd f^{(k)}$ by information obtained from previous iterates. This requires the \emph{parallel transport} of tangent vectors in $T_{x^{(k-1)}} \mc{M}$ to vectors in $T_{x^{(k)}} \mc{M}$. We refer e.g. to \cite[Section 3]{Jost2005} for details. We take the parallel transport as the differential of the $e$-Exp map
\begin{equation*}
    P^e_{x \to x^\prime}: T_x\mc{M} \to T_{x^{\prime}} \mc{M} \qquad v \mapsto d \Exp_x(u)v =  g(x^\prime)^{-1} g(x) v = \frac{x^\prime}{x} v,
\end{equation*}
for $x^\prime = \Exp_x(u)$, cf. \cite[Lemma 4.3]{Raus2024}. The CG-based accelerated EG update writes
\begin{subequations} \label{eq:Poi-CG}
    \begin{align}
        x^{(k+1)} &= \Exp_{x^{(k)}}^e(\tau_k v^{(k)}) = x^{(k)} \exp(\tau_k v^{(k)}) \\
        v^{(k+1)} &= - \Rd f(x^{(k+1)}) + \beta_{k+1} P_{x^{(k)} \to x^{(k+1)}}^e(v^{(k)}) \\
        &= - \Rd f(x^{(k+1)}) + \beta_{k+1} v^{(k)} \exp(\tau_k v^{(k)})
    \end{align}
\end{subequations}
where $\tau_k > 0$ is the step size, and $\beta_{k+1} \in \R$ is a suitably chosen parameter. Among the existing choices for $\beta_{k+1}$ we achieved our best numerical results for the geometric version of Polak-Ribiere (PR) \cite{Polak_Ribiere_1969} type CG
\begin{equation}\label{eq:Poi-PR}
    \beta^{\text{PR}}_{k+1} = \frac{\la \Rd f^{(k+1)}, u^{(k)} \ra_{x^{(k+1)}}}{\| \Rd f^{(k)} \|^2_{x^{(k)}}},
\end{equation}
where $u^{(k)} := \Rd f(x^{(k+1)}) - v^{(k)} \exp(\tau_k v^{(k)})$.
Global convergence of geometric CG methods with various choices for the parameter $\beta$ has been extensively studied; see \cite[Table 1]{Sakai_Sato_Iiduka_2023} for an overview. Notably, this includes a hybrid PR-type method analyzed in \cite{Sakai_Iiduka_2021}. However, existing convergence results for geometric CG methods rely on Lipschitz-type assumptions on the cost function $f$ and employ strong Wolfe conditions for inexact line searches. 
The convergence analysis of geometric CG-based accelerated EG under weaker assumptions remains an open question, which we address in future work.

\section{Experiments}
\label{sec:experiments}
\begin{figure}[ht]
\centering
\begin{minipage}[t]{0.15\columnwidth}
    \includegraphics[width=1.05\textwidth]{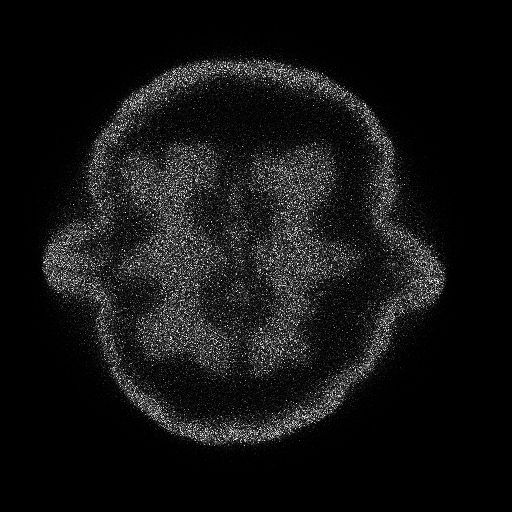}
\end{minipage}%
\hspace{0.1pt}
\begin{minipage}[t]{0.15\columnwidth}
    \includegraphics[width=1.05\textwidth]{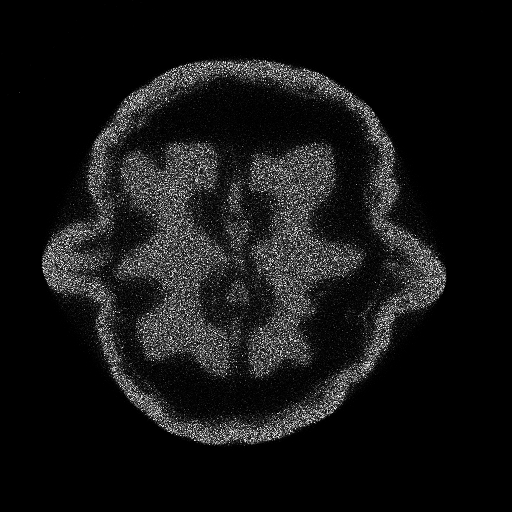}
\end{minipage}%
\hspace{0.1pt}
\begin{minipage}[t]{0.15\columnwidth}
    \includegraphics[width=1.05\textwidth]{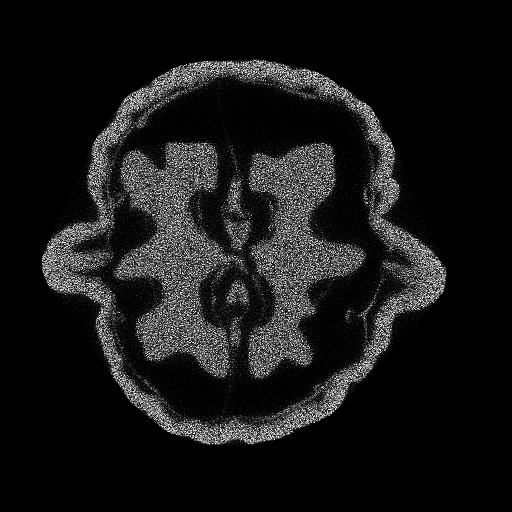}
\end{minipage}%
\hspace{0.1pt}
\begin{minipage}[t]{0.15\columnwidth}
    \includegraphics[width=1.05\textwidth]{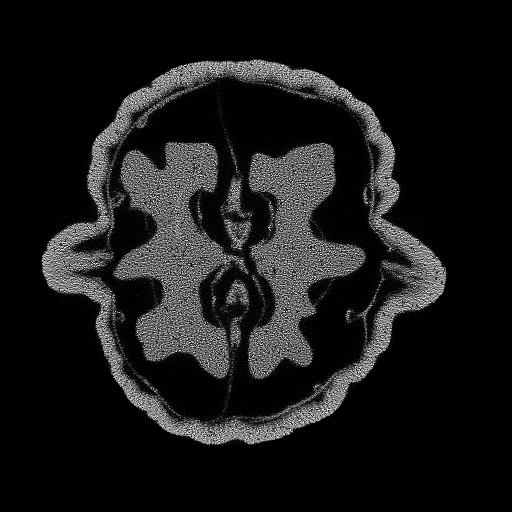}
\end{minipage}%
\hspace{0.1pt}
\begin{minipage}[t]{0.15\columnwidth}
    \includegraphics[width=1.05\textwidth]{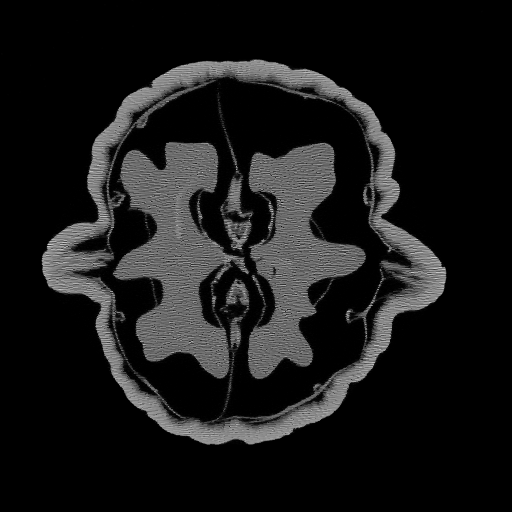}
\end{minipage}

\vspace{0.15cm}

\begin{minipage}[t]{0.15\columnwidth}
    \includegraphics[width=1.05\textwidth]{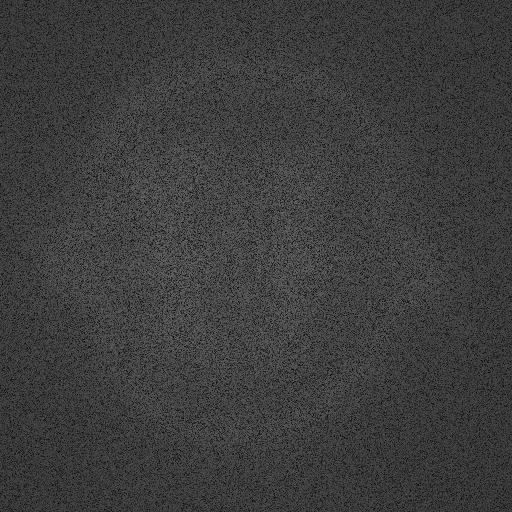}
\end{minipage}%
\hspace{0.1pt}
\begin{minipage}[t]{0.15\columnwidth}
    \includegraphics[width=1.05\textwidth]{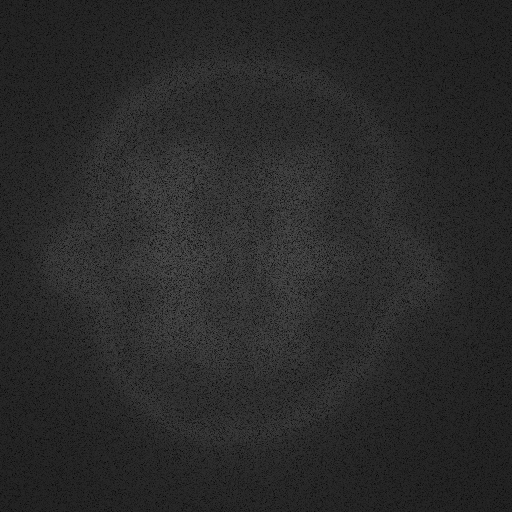}
\end{minipage}%
\hspace{0.1pt}
\begin{minipage}[t]{0.15\columnwidth}
    \includegraphics[width=1.05\textwidth]{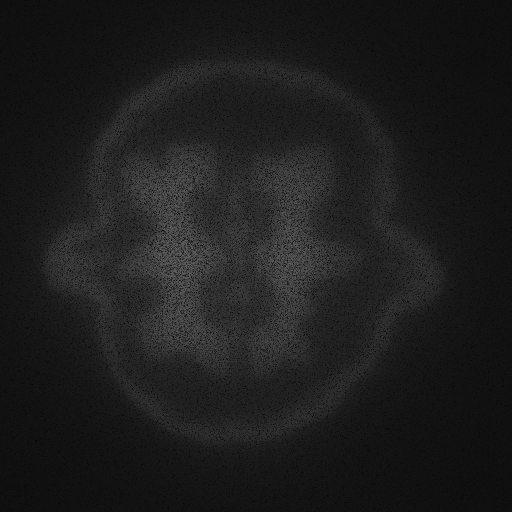}
\end{minipage}%
\hspace{0.1pt}
\begin{minipage}[t]{0.15\columnwidth}
    \includegraphics[width=1.05\textwidth]{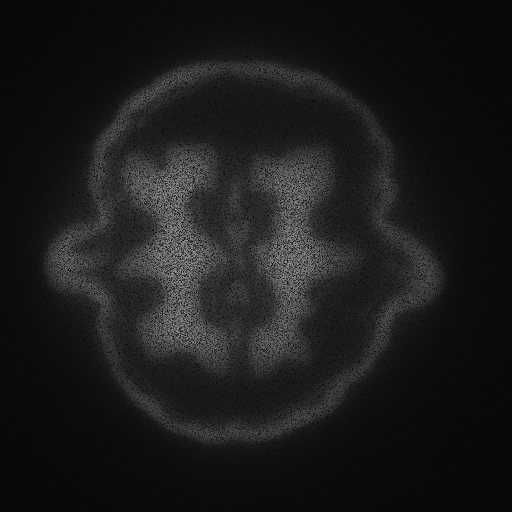}
\end{minipage}%
\hspace{0.1pt}
\begin{minipage}[t]{0.15\columnwidth}
    \includegraphics[width=1.05\textwidth]{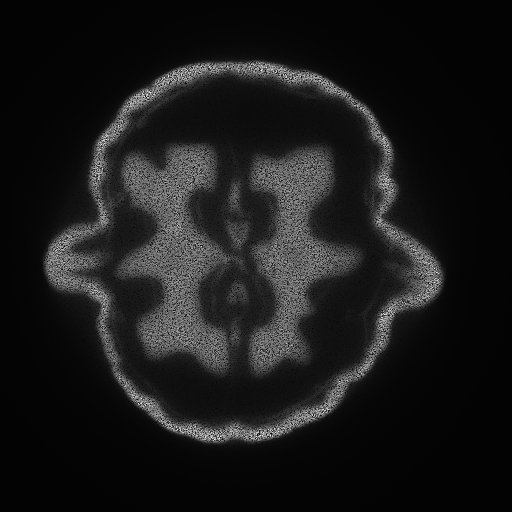}
\end{minipage}

\begin{minipage}[t]{0.15\columnwidth}
    \includegraphics[width=1.05\textwidth]{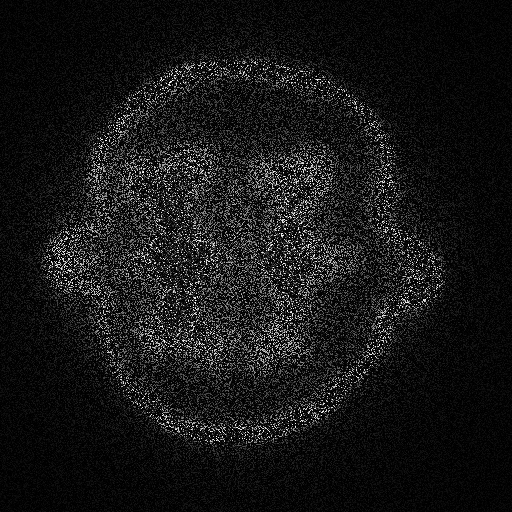}
\end{minipage}%
\hspace{0.1pt}
\begin{minipage}[t]{0.15\columnwidth}
    \includegraphics[width=1.05\textwidth]{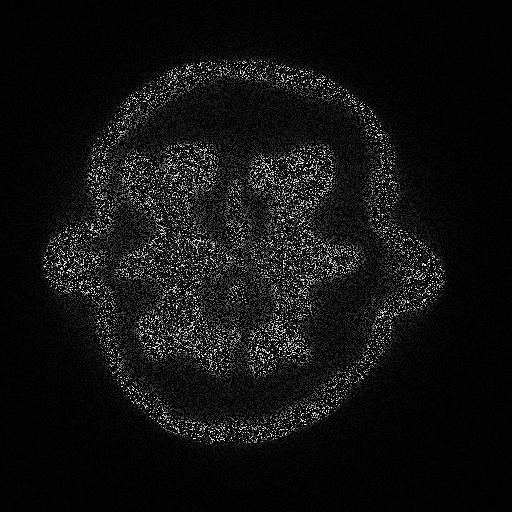}
\end{minipage}%
\hspace{0.1pt}
\begin{minipage}[t]{0.15\columnwidth}
    \includegraphics[width=1.05\textwidth]{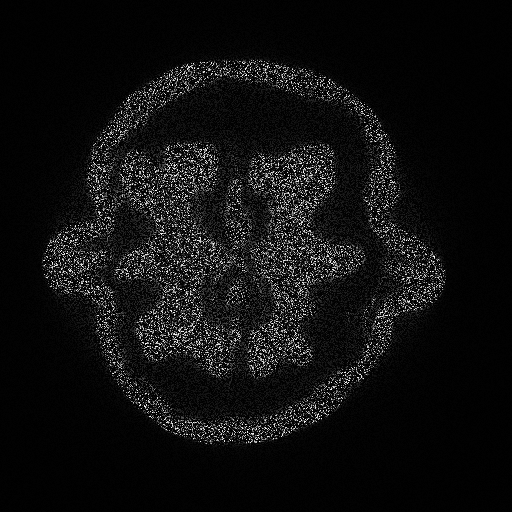}
\end{minipage}%
\hspace{0.1pt}
\begin{minipage}[t]{0.15\columnwidth}
    \includegraphics[width=1.05\textwidth]{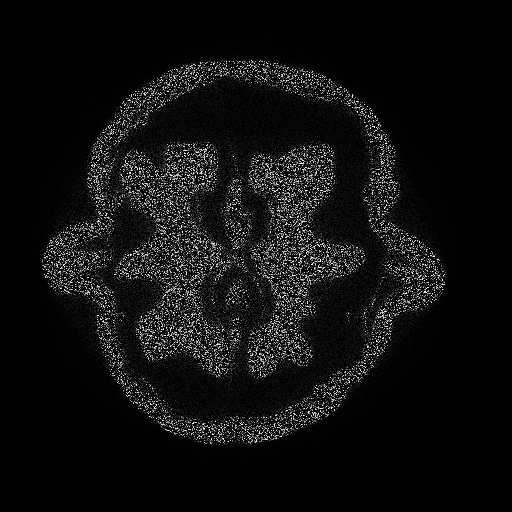}
\end{minipage}%
\hspace{0.1pt}
\begin{minipage}[t]{0.15\columnwidth}
    \includegraphics[width=1.05\textwidth]{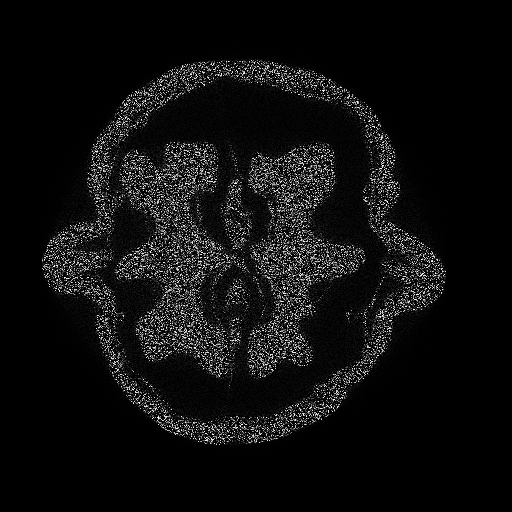}
\end{minipage}

\vspace{0.02cm} 

\begin{minipage}[t]{0.15\columnwidth}
    \includegraphics[width=1.05\textwidth]{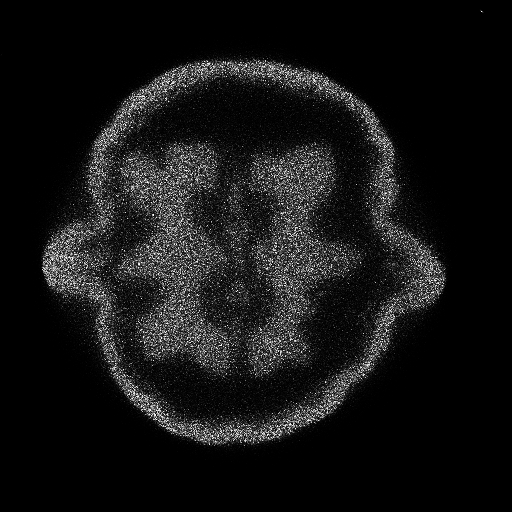}
\end{minipage}%
\hspace{0.1pt}
\begin{minipage}[t]{0.15\columnwidth}
    \includegraphics[width=1.05\textwidth]{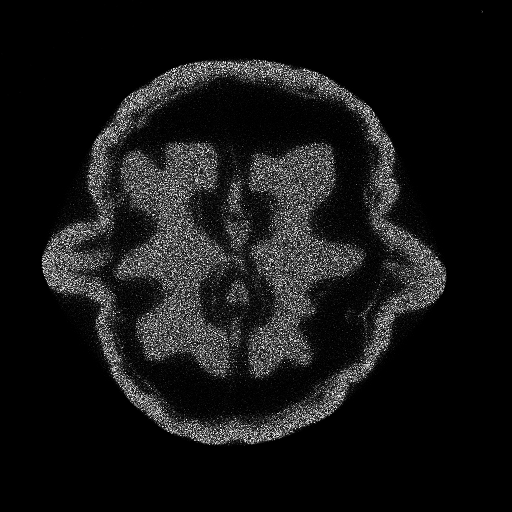}
\end{minipage}%
\hspace{0.1pt}
\begin{minipage}[t]{0.15\columnwidth}
    \includegraphics[width=1.05\textwidth]{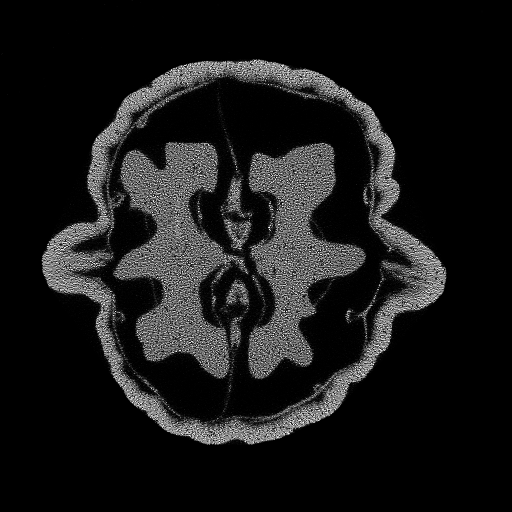}
\end{minipage}%
\hspace{0.1pt}
\begin{minipage}[t]{0.15\columnwidth}
    \includegraphics[width=1.05\textwidth]{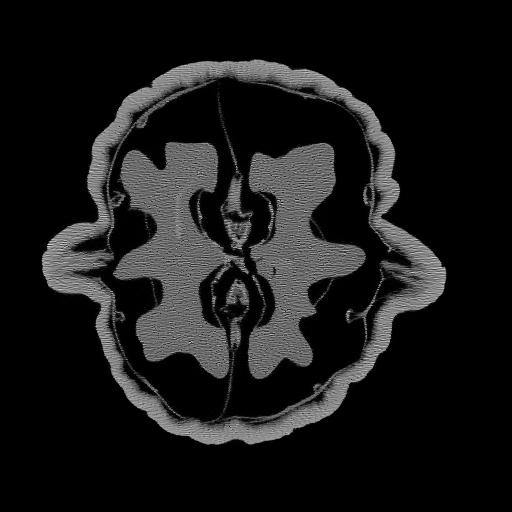}
\end{minipage}%
\hspace{0.1pt}
\begin{minipage}[t]{0.15\columnwidth}
    \includegraphics[width=1.05\textwidth]{Poi_CG_e_iterations_100.png}
\end{minipage}

\caption{\textbf{Reconstructions across iterations.} Comparison of \textbf{EG} with Armijo backtracking (top row), $g$-RGD with Armijo backtracking (second row), and IP $e$-RGD with a constant step size (third row) based on interior-point geometry, along with PR-type CG-accelerated EG for iterations $k = 10, 20, 50, 100, 300$. 
Since accelerated EG terminates at 148 iterations, its result at iteration 148 is shown for iteration 300.
}
\label{fig:walnut-reconstruction}
\end{figure}

We consider the reconstruction problem
\begin{equation*}
    f(x) = \min_{x \in \mc{M}} \KL(b,Ax) + \lambda \la L_\delta(\nabla x), \eins \ra,
\end{equation*}
which violates $L$-smoothness.  $L_{\delta}$ is the Huber function defined by $L_{\delta}(a)=\frac{1}{2}a^{2}$ for $|a|\leq\delta$ and $L_{\delta}(a)=\delta\cdot\left(|a|-\frac{1}{2}\delta\right)$ otherwise. The discrete gradient $\nabla x$ is the vector of finite differences.
We aim to recover a discretized signal $x$ from linear positive measurements $b \in {\R}^m_{++}$ using a nonnegative matrix $A \in {\R}^{m \times n}_{+}$. The cost function $f$ and its gradient are neither Lipschitz nor relatively Lipschitz smooth with respect to the negative entropy $\psi$ \eqref{eq:def-negative-entropy}, making standard MD convergence guarantees inapplicable. However, since $f$ is (Euclidean) convex, \cref{cor:convergence-loc-min} ensures global convergence of EG with Riemannian Armijo.

\noindent
\textbf{Problem setup.} 
We consider tomographic reconstruction as a test case, specifically reconstructing the walnut phantom at a resolution of $512 \times 512$ \cref{fig:walnut-reconstruction}. The phantom is flattened into a one-dimensional vector for processing. The tomographic projection matrices $A$ are generated using the ASTRA toolbox \footnote{\url{https://astra-toolbox.com/}} with parallel-beam geometry and equidistant angles in the range $[0, 2 \pi]$. The undersampling rate is set to $20\%$.

We interpret EG  as the $e$-RGD in Poisson geometry with Armijo backtracking and compare it to two other geometry-based iterative methods on the positive orthant:
\begin{enumerate}
    \item \textit{g-RGD on $(\mc{M}, g_{\superIP})$}: Implemented as a scaled EG iteration, see iteration definition \eqref{eq:IP-g-RGD}, with backtracking line search.
    \item \textit{e-RGD on $(\mc{M}, g_{\superIP})$}: Formulated as a mirror descent iteration  \cite[Proposition 4.6]{Raus2024} for the Bregman function $\vphi(x) = - \la \log x, \eins \ra$, gives:
    \begin{equation} \label{eq:MD-log-barrier}
        x^{(k+1)} = \frac{x^{(k)}}{1 - \tau_k x^{(k)} \nabla f^{(k)}}.
    \end{equation}
    In this case, $\vphi^\ast$ corresponds to the log partition function of the exponential distribution.
    Since $f$ is relative Lipschitz smooth with respect to $\vphi$, convergence of \eqref{eq:MD-log-barrier} to a global minimum given the convexity of $f$ is guaranteed for a constant step size $\tau_k \leq \frac{1}{2L}$, with $L = \|b\|_1$ as the relative Lipschitz constant \cite{Bauschke:2017aa}. Consequently, we use a constant step size for this evaluation.
    \item \textit{Poisson e-CG}: Using Armijo-line search for the step sizes $\tau_k$ and PR-type choice for $\beta_{k+1}$, s. \eqref{eq:Poi-CG} and \eqref{eq:Poi-PR}.
\end{enumerate}
\textbf{Implementation details.} 
We implemented the Poisson and interior point geometries, including their respective exponential maps, as well as the RGD and CG methods using the framework provided by the Python library Pymanopt \cite{Townsend16}. All algorithms start with the same random initialization $x^{(0)} \in \mc{M}$. 
The termination criteria follow the default settings of the Pymanopt library: a maximum of 300 iterations, a minimum gradient norm of $10^{-6}$, and a minimum step size of $10^{-10}$ for backtracking line search. In our experiments accelerated EG occasionally terminates due to reaching the minimum step size. The maximum number of iterations is typically reached first for all RGD-based algorithms.

\begin{figure}[H]
    \centering
    \begin{tabular}{cc} 
        \includegraphics[width=0.55\textwidth]{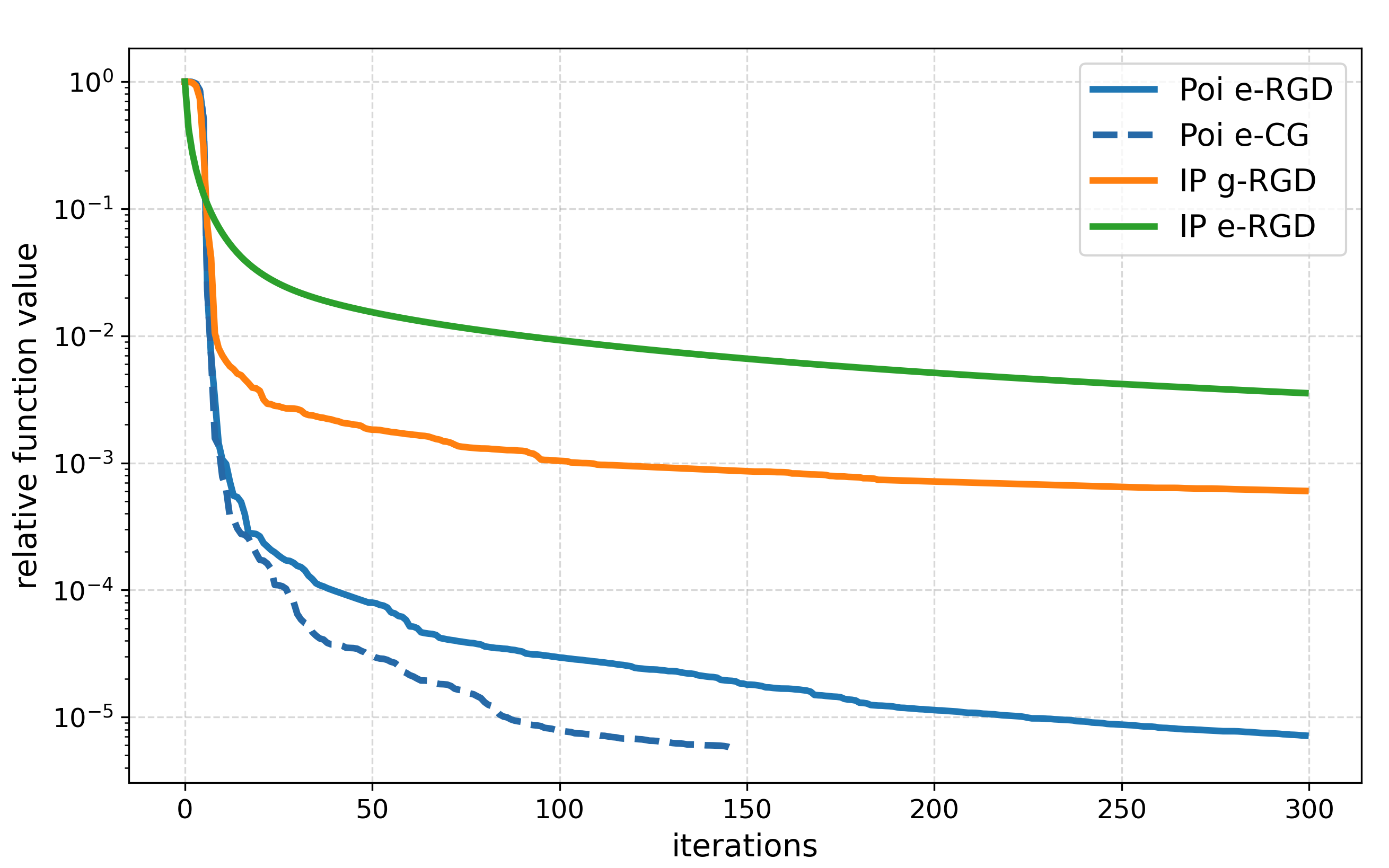}&
        \includegraphics[width=0.45\textwidth]{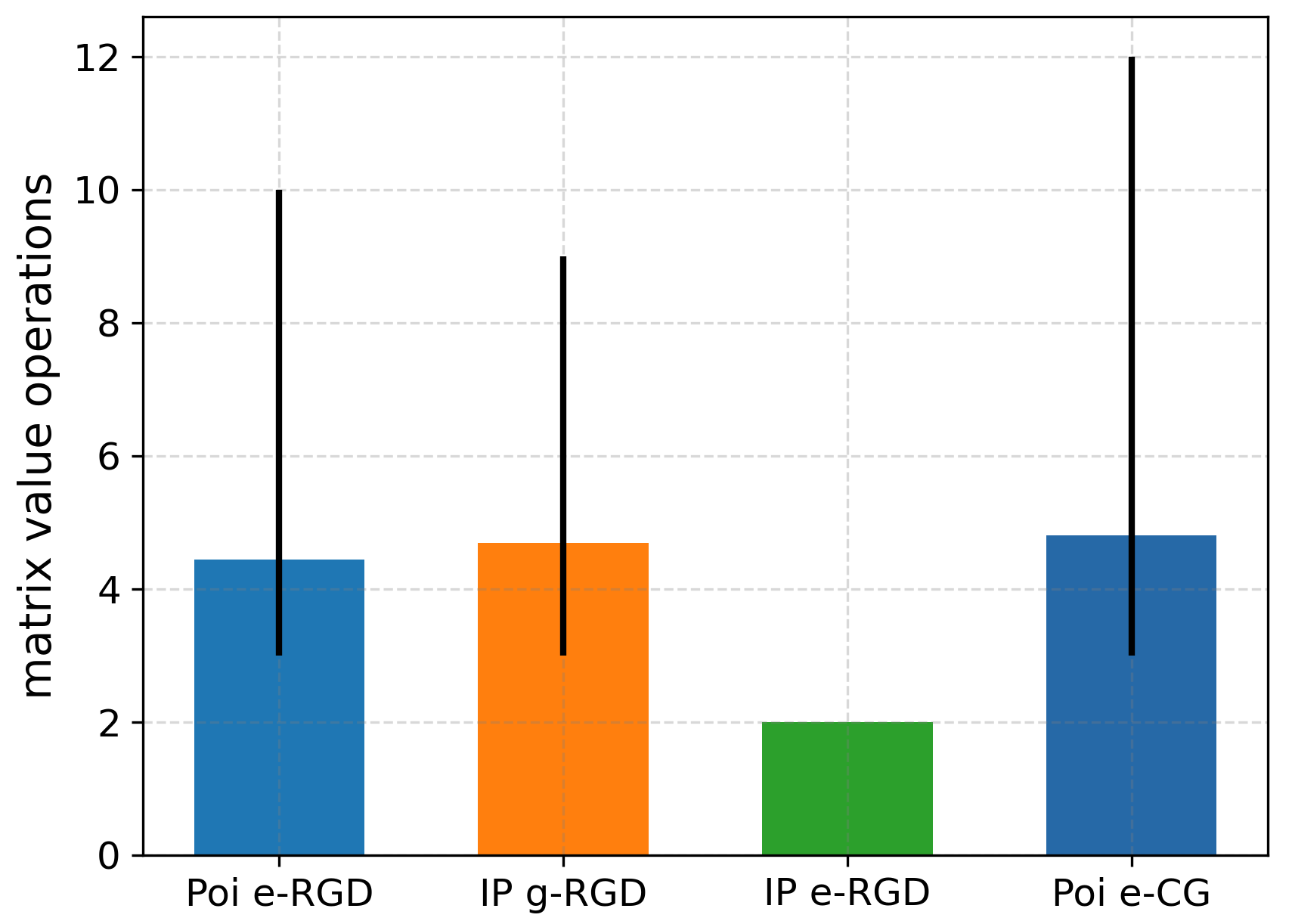} \\ 
    \end{tabular}
    \caption{\textbf{Comparison of Algorithms.}
\textsc{Left:} Relative function values of EG (Poi $e$-RGD), its accelerated variant (Poi $e$-CG), and $g$-RGD and $e$-RGD with interior-point geometry. Accelerated EG outperforms and terminates after 148 iterations.
\textsc{Right:} Average matrix-value operations across iterations. While IP $e$-RGD is cheaper without line search, EG still achieves better performance.}
    \label{fig:two_rows_images}
\end{figure}
\section{Conclusion}\label{sec:conclusion}
We characterized the exponentiated gradient (EG) method as Riemannian gradient descent (RGD) on the Poisson parameter manifold induced by Fisher-Rao geometry, with a focus on Riemannian line search via a suitable retraction. We proved finite termination under weak conditions beyond standard $L$-smoothness. Our setup enabled efficient vector transport and motivated a conjugate EG method, whose convergence under similar conditions is left for future work. Numerical experiments, including an accelerated variant, highlight EG's practical advantages, such as faster convergence compared to interior-point RGD.

\vspace{2mm}
\textbf{Acknowledgment.} This work is funded by Deutsche Forschungsgemeinschaft (DFG) under Germany’s Excellence Strategy EXC-2181/1 - 390900948 (the Heidelberg STRUCTURES Excellence Cluster).

\FloatBarrier

\bibliographystyle{plain}
\bibliography{bibliography}
\end{document}